\documentclass[11pt,reqno]{amsart}

\usepackage[a4paper,left=30mm,right=30mm,top=30mm,bottom=30mm,marginpar=25mm]{geometry}
\usepackage{amsmath, amssymb}
\usepackage{amsfonts}
\usepackage{amsthm}
\usepackage{amsthm}
\usepackage{mathrsfs}
\usepackage{comment}
\usepackage{color}
\usepackage{enumerate}

\usepackage[nocompress, space]{cite}
\usepackage[displaymath,mathlines]{lineno}
\usepackage{eurosym}
\usepackage[dvips]{graphics}
\usepackage{graphicx}
\usepackage{epsfig}
\usepackage{mathtools}

\usepackage{hyperref}
\usepackage[nocompress, space]{cite}
\usepackage{caption}
\usepackage{subcaption}

\newcommand{\al}{\alpha}
\newcommand{\be}{\beta}
\newcommand{\de}{\delta}

\newcommand{\ep}{\varepsilon}
\newcommand{\ph}{\varphi}
\newcommand{\pa}{\partial}
\newcommand{\Ga}{\Gamma}
\newcommand{\na}{\nabla}
\newcommand{\dis}{\displaystyle}
\newcommand{\R}{\mathbb{R}}

\newcommand{\N}{\mathbb{N}}
\def\coloneqq{\mathrel{\mathop:}=}%
\def\revcoloneqq{=\mathrel{\mathop:}}%

\newcommand{\AC}{{\rm AC\,}}

\theoremstyle{definition}

\newtheorem{Th}{Theorem}[section]
\newtheorem{Prop}[Th]{Proposition}
\newtheorem{Lem}[Th]{Lemma}
\newtheorem{Def}{Definition}
\newtheorem{Rem}{Remark}

\numberwithin{equation}{section}

\title[Equivalence of two weak solutions]{On the equivalence of \\viscosity solutions and distributional solutions\\
 for the time-fractional diffusion equation}

\author{Yoshikazu Giga, Hiroyoshi Mitake, Shoichi Sato}
\address{Graduate School of Mathematical Sciences, University of Tokyo, 3-8-1 Komaba,
Meguroku, Tokyo, Japan 153-8914}
\email{labgiga@ms.u-tokyo.ac.jp, mitake@ms.u-tokyo.ac.jp, shoichi@ms.u-tokyo.ac.jp}
\thanks{
        Y. G. was partially supported by the JSPS grants: KAKENHI \#19H00639, \#18H05323, \#17H01091,   %
                H. M. was partially supported by the JSPS grants: KAKENHI \#19K03580, \#17KK0093, \#20H01816.  
}
\keywords{Caputo time fractional derivatives, Initial-boundary value problems, Fractional diffusion equation, Viscosity solutions, Distoributional solutions.}
\subjclass[2010]{
34K37, 
35D30, 
49L25} 
\date{\today}

\begin{document}
\maketitle

\begin{abstract}
We consider an initial-boundary value problem for the time-fractional diffusion equation.
We prove the equivalence of two notions of weak solutions, viscosity solutions and distributional solutions.
\end{abstract}

\section{Introduction}

Let $T>0$ and $\al \in (0, 1)$ be given constants and $\Omega \subset \R^d$ be a bounded domain with smooth boundary.
We are concerned with the time-fractional diffusion equation: 
\begin{equation} \label{CP}
\left \{
\begin{array}{ll}
d_{t}^{\al}u(x,t) + Lu(x,t) \ =  f(x,t) & \mbox{for} \quad (x,t)\in\Omega \times (0,T), \\
u(x,t) =  0 & \mbox{for} \quad (x,t)\in\pa\Omega \times (0,T), \\
u(x,0) = u_{0}(x) & \mbox{for} \quad x\in \overline{\Omega}.
\end{array}
\right.
\end{equation}
Here, $u:\overline{\Omega} \times [0,T] \rightarrow \R$ is an unknown function, $f:\overline{\Omega} \times [0,T] \rightarrow \R$, $u_0:\overline{\Omega}\to\mathbb{R}$ are given continuous functions and $L$ is a symmetric uniformly elliptic operator of divergence form given by 
\begin{equation*}
Lu(x, t) \coloneqq - \sum_{i, j = 1}^{d} \pa_i (a_{ij}(x) \pa_j u(x, t)), 
\end{equation*}
where $(a_{ij})$ is a given continuous diffusion coefficient satisfying the following assumptions: 
\begin{enumerate}
\item[(A1)]
 The matrix $(a_{ij}(x))_{ij}$ is symmetric, i.e.,
\begin{equation*}
a_{ij}(x) = a_{ji}(x) \quad \mbox{for all} \quad i,j = 1, \dots, d, \ \text{and} \ x \in \Omega. 
\end{equation*}
\item[(A2)]
 The operator $L$ is uniformly elliptic, i.e., there exists a constant $\lambda\in(0,1)$ such that
\begin{equation*}
\lambda|\xi|^2 \leq \sum_{i,j=1}^{d} a_{ij}(x) \xi_i \xi_j
\quad \mbox{for all} \quad \xi \in \R^d \ \text{and} \ x \in \Omega. 
\end{equation*}
\end{enumerate}
Moreover, we denote by $d_{t}^{\al}u$ the \textit{Caputo fractional derivative} of $u$ with respect to $t$, that is, 
\begin{equation*}
d_{t}^{\al}u(x,t) \coloneqq \dis \left( g_{1-\al} \ast \frac{du}{dt} \right)(x,t)
=\int_0^{t}g_{1-\al}(t-s)\frac{du}{dt}(x,s)\,ds
\end{equation*}
for all $(x,t)\in\overline{\Omega}\times(0,T)$, where we write $g_{\be}$ for the Riemann-Liouville kernel, 
\begin{equation*}
\dis g_{\be}(t) \coloneqq \frac{t^{\be-1}}{\Ga(\be)} \quad\text{for} \  t>0 \ \text{and} \ \be >0, 
\end{equation*}
where $\Gamma$ is the Gamma function. 

\medskip
Fractional derivatives attracted great interest from both mathematics and applications within the last few decades, and developed in wide fields (see \cite{MK, NSY, KRY, KST, P} for instance).  Studying differential equations with fractional derivatives is motivated by mathematical models that describe diffusion phenomena in complex media like fractals, which is sometimes called \textit{anomalous diffusion}. It has inspired further research on numerous related topics. 
We refer to a non-exhaustive list of references \cite{SY, ACV, GN, TY, N, KY, CKKW, LTY, CG} and the references therein.  

The well-posedness of the initial-boundary value problem \eqref{CP} was first established by using the Galerkin method in \cite{Z, SY} in the framework of distributional solutions.  
Also, the existence of unique viscosity solutions to \eqref{CP} was established by \cite{GN, N}, and also by \cite{TY} in the whole space. 
It is worth emphasizing that as far as the authors know the relations between two weak solutions, 
a viscosity solution and a distributional solution, are not studied yet. 
The objective of our paper is to prove the equivalence of two notions of weak solutions. 

\subsection{Definitions of two weak solutions} 
Here, we recall the definitions of the viscosity solution and the distributional solution to \eqref{CP}. 

\begin{Def}\label{vissol}
A upper semicontinuous function $u:\overline{\Omega} \times [0, T)\to\mathbb{R}$
is said to be a \textit{viscosity subsolution} of \eqref{CP} if  
for any $\ph \in C^2(\overline{\Omega} \times [0, T])$ one has
\begin{equation*}
d_{t}^{\al} \ph(x_0,t_0) + L\ph (x_0,t_0) \leq \, f(x_0, t_0), 
\end{equation*}
whenever $u-\ph$ attains a local maximum at $(x_0,t_0) \in \Omega \times (0,T)$, $u(x,t) \leq 0$ for all $(x,t)\in\pa \Omega \times (0,T)$, and $u(x,0) \leq u_0(x)$ for all $x\in\overline{\Omega}$.

Similarly, a lower semicontinuous function $u:\overline{\Omega} \times [0, T)\to\mathbb{R}$
is said to be a \textit{viscosity supersolution} of \eqref{CP} if  
for any $\ph \in C^2(\overline{\Omega} \times [0, T])$ one has
\begin{equation*}
d_{t}^{\al} \ph(x_0,t_0) + L\ph (x_0,t_0) \geq \, f(x_0, t_0), 
\end{equation*}
whenever $u-\ph$ attains a local minimum at $(x_0,t_0) \in \Omega \times (0,T)$, $u(x,t) \geq 0$ for all $(x,t)\in\pa \Omega \times (0,T)$, and $u(x,0) \geq u_0(x)$ for all $x\in\overline{\Omega}$.

Finally, we call $u\in C(\overline{\Omega}\times[0,T))$ a \textit{viscosity solution} 
of \eqref{CP} if $u$ is both a viscosity subsolution and supersolution of \eqref{CP}. 
\end{Def}

Thanks to \cite{N} it is known that under (A1), and the assumptions 
\begin{align*}
&\sum_{i,j=1}^{d}a_{ij}(x)\xi_i\xi_j\ge0 \quad\text{for all} \ x\in\overline{\Omega}, \ \xi\in\mathbb{R}^d, \\
&a_{ij}\in C^{1,1}(\overline{\Omega}\times[0,T]), \ f\in C(\overline{\Omega}),  \ \text{and} \ u_0\in C(\overline{\Omega}) \ 
 \text{with} \ u_0=0 \ \text{on} \ \partial\Omega, 
\end{align*}
\eqref{CP} has the unique viscosity solution. 

\medskip
To define the distributional solution we first notice that we have the relation 
(see Lemma \ref{RL=Cap} for the proof) 
\begin{equation} \label{C=RL-0}
d_t^{\al}v(t)=D_t^{\al}(v-v(0))(t) \quad\text{for all} \  t\in(0,T)
\end{equation}
for all $v\in\AC(0,T)$, which denotes the set of all absolutely continuous functions. 
Here, we write $D_t^{\al}v$ for the \textit{Riemann-Liouville fractional derivative} of $v$ with order $\alpha$, that is, 
\begin{equation*}
D_{t}^{\al}v(t) \coloneqq \dis \frac{d}{dt} (g_{1-\al} \ast v)(t). 
\end{equation*}
Thus, we can rewrite \eqref{CP} as 
\begin{equation} \label{RLP} 
\left \{
\begin{array}{ll}
D_{t}^{\al}(u-u_0)(x,t) + Lu(x,t) =  f(x,t) & \mbox{for} \quad (x,t)\in \Omega \times (0,T), \\
u(x,t) =  0 & \mbox{for} \quad (x,t)\in\pa\Omega \times (0,T), \\
u(x,0) = u_{0}(x) & \mbox{for} \quad x\in\overline{\Omega}.
\end{array}
\right.
\end{equation}

Let $H^1(\Omega)$ be the standard Sobolev space, $W^{1,2}(\Omega)$, and 
$H_0^1(\Omega)$ be the space of functions in $H^1(\Omega)$  
that vanish at the boundary in the sense of traces. 
Also, we denote by $L^2(0,T; H_0^1(\Omega))$ the parabolic Sobolev space 
(see \cite{E} for the notation).  
We set 
\begin{equation*}
W^{\al}(u_0) \coloneqq \big\{ u\in L^2(0,T; H_0^1(\Omega)) \mid \, g_{1-\al} \ast (u-u_0) \in {}_{0}H^1(0,T; H^{-1}(\Omega)) \big\},
\end{equation*}
where we denote by ${}_{0}H^1(0,T; H^{-1}(\Omega))$ the set of all functions in $H^1(0,T; H^{-1}(\Omega))$ whose trace on $t=0$ is zero.

\begin{Def}
We call $u$ a \textit{distributional solution} to (\ref{RLP}) if $u \in W^{\al}(u_0)$
and $u$ satisfies 
\begin{multline} \label{wf}
\dis \frac{d}{dt} \int_{\Omega} [g_{1-\al} \ast (u-u_0)](x,t) \ph(x) \,dx
+ \sum_{i,j=1}^{d}\int_{\Omega} a_{ij}(x) \pa_{x_j} u(x,t) \pa_{x_i} \ph(x) \,dx \\
= 
\int_{\Omega}f(x,t)\varphi(x)\,dx
\end{multline} 
for all $\ph \in H_0^1(\Omega)$ and almost every $t \in (0, T)$. 
\end{Def}

This is a special case of the definition which was introduced in \cite{Z}. 
Thanks to \cite{Z,SY}, it is known that under (A1), (A2) and the assumptions 
\[
a_{ij}\in C^1(\overline{\Omega}), f\in L^{\infty}(0,T;L^2(\Omega)), \ \text{and} \ u_0\in L^2(\Omega), 
\]
\eqref{RLP} has the unique distributional solution. 

\subsection{Main Result} 
We state our main result in the paper. 
\begin{Th} \label{Main_Th}
Let $u\in C(\overline{\Omega}\times[0,T))$. 
Assume that (A1), (A2), 
\begin{enumerate}
\item[(A3)]  $a_{ij} \in C^{1, 1}(\overline{\Omega})$, $f \in C(\overline{\Omega} \times [0, T])$,  
\item[(A4)] $u_0\in C(\overline{\Omega})$ with $u_0=0$ on $\partial\Omega$, 
\item[(A5)] $\Omega\subset\mathbb{R}^d$ is a bounded domain with $C^2$-boundary  
\end{enumerate}
hold. Then $u$ is the viscosity solution of \eqref{CP} if and only if $u$ is the distributional solution of \eqref{RLP}. 
\end{Th}

\medskip
Let us briefly describe our approaches to get Theorem \ref{Main_Th}. First it is worth emphasizing that in general the notion of viscosity solutions is based on the comparison principle, while the notion of distributional solutions is based on the variational principle. Since two notions of weak solutions are introduced in totally different manners, it is highly nontrivial whether two notions are same in the settings under (A1)--(A4) or not. 
In our approach we use the discrete scheme for time fractional diffusion equations which was introduced in \cite{GLM}. This scheme can be regarded as a resolvent-type approximation (see \cite{GLM} for details). 
In \cite{GLM}, in a rather general setting, it is proved that an approximated solution uniformly converges to a viscosity solution to \eqref{CP}. In this paper, we modify this approximation to make it an absolutely continuous function and prove that it converges to a distributional solution to \eqref{RLP} in a suitable norm. 

A main difficulty is in proving that the error term which comes from the approximated solution and the distributional solution converges to zero in a suitable weak sense. 
Our approach here is to introduce an approximation of kernel $g_{1-\alpha}$ in consideration of the discrete scheme, which is our key gradient of our paper. 
Due to the discrete scheme and kernel approximation \eqref{app:kernel}, we can get the precise error estimate which enables us  to get our main theorem, Theorem \ref{Main_Th}. 

\medskip
We conclude this introduction by giving a non exhaustive list of related works to our paper. 
The regularity of solutions to a space-time nonlocal equation with Caputo's time fractional derivative is studied in \cite{ACV, CG}. 
The large time behavior of viscosity solution to Hamilton--Jacobi equations with Caputo time derivative under the periodic boundary condition is studied. 
The well-posedness and the representation to the weak solution are given in \cite{CKKW} from the probabilistic point of view. 

Results on various analytical aspects of time-fractional diffusion equations are summarized in \cite{ZZ}.
For example, recent developments towards the De Giorgi-Nash-Moser theory, and the large time behavior of distributional solutions are surveyed.
The large time behavior of distributional solutions of the evolution equation with Caputo time derivative in a bounded domain is studied in \cite{DVV}.
Since this approach is energy estimates in an abstract framework, the spatial operator is not limited to a linear elliptic differential operator but can be also taken as fractional elliptic operators as well as nonlinear elliptic operators.

We also give several results on the equivalence of two notions of weak solutions. 
In \cite{L, I} the equivalence is studied for linear degenerate elliptic equations. 
In \cite{JLM, JJ}, it is studied for  $p$-Laplace equations. 

\bigskip
This paper is organized as follows. 
In Section \ref{sec:DS}, we recall the discrete scheme introduced in \cite{GLM} and modify it to adjust it for our purpose. 
In Section \ref{sec:KA}, we introduce a kernel approximation and Section \ref{sec:EE} is devoted to give the energy estimate which is a key ingredient of the paper. 
We finally give a proof of Theorem \ref{Main_Th} in Section \ref{sec:main-proof}.

\section{The discrete scheme}\label{sec:DS}
In this section, we first quickly recall the definition of the discrete scheme which was first introduced in \cite{GLM}.
Let $T > 0, M \in \N$ and set $h \coloneqq T/M$.
Note that
\begin{equation*}
d_{t}^{\al}u(mh) =
\dis \int_{0}^{mh} g_{1-\al}(mh-s) \frac{du}{ds}(s) \,ds =
\dis \sum_{k=0}^{m-1} \int_{kh}^{(k+1)h} g_{1-\al}(mh-s) \frac{du}{ds}(s) \,ds
\end{equation*}
for any function $u \in \AC([0, T])$, and $m \in \mathbb{N}$.
If $u$ is a smooth function in $\Omega \times (0, T)$ and $h$ is sufficiently small, then we are able to approximate as 
\begin{equation*}
\dis \int_{kh}^{(k+1)h} g_{1-\al}(mh-s) \frac{du}{ds}(s) \,ds
\fallingdotseq
\dis \int_{kh}^{(k+1)h} g_{1-\al}(mh-s) \frac{u((k+1)h)-u(kh)}{h} \,ds.
\end{equation*}
Here, we have
\begin{align*}
\dis \int_{kh}^{(k+1)h} g_{1-\al}(mh-s) \,ds &= g_{2-\al}((m-k)h) - g_{2-\al}((m-k-1)h) \\
&= (g_{2-\al}(m-k) - g_{2-\al}(m-k-1))h^{1-\al} \\
&= \psi(m-k)h^{1-\al},
\end{align*}
where we set 
\begin{equation*}
\psi(r) \coloneqq g_{2-\al}(r) - g_{2-\al}(r-1) \quad \mbox{for} \quad r \geq 1.
\end{equation*}
Note that $\psi'(r)\le0$ for all $r\ge1$. 
With this observation, we heuristically have 
\begin{align}
d_{t}^{\al}u(mh) &\fallingdotseq \dis \frac{1}{h^{\al}} \sum_{k=0}^{m-1} \psi(m-k) (u((k+1)h) - u(kh)) \nonumber\\
&= \dis \frac{1}{\Ga(2-\al)h^{\al}} \left\{ u(mh) - \sum_{k=0}^{m-1} C_{m,k}u(kh) \right\}, 
\label{def:dCD}
\end{align}
where
\begin{equation*}
C_{m,k} \coloneqq 
\left\{
\begin{array}{ll}
\Ga(2-\al)\psi(m) \quad & \mbox{for} \quad k = 0 \\
\Ga(2-\al)(\psi(m-k) - \psi(m-(k-1))) \quad &\mbox{for} \quad k = 1, \ldots,m-1.
\end{array}
\right. 
\end{equation*}
Set $\widetilde{M}:=\{0,1, \cdots, M \}$. 
It is important to notice here that we have 
\[
C_{m,k}\ge 0 \quad \text{for all} \ k=0,1,\ldots,m-1, \ \text{and} \ m\in\widetilde{M}\setminus\{0\},  
\]
since $\psi$ is non-increasing. 

Take
\begin{equation} \label{ini_app}
\displaystyle U_0^h \in C_c^{\infty}(\Omega) \cap C(\overline{\Omega}) \, \, \mbox{so that} \, \, \sup_{\overline{\Omega}} |U_0^h - u_0| \rightarrow 0 \, \, \mbox{as} \, \, h \rightarrow 0.
\end{equation}
For $m \in \widetilde{M}\setminus\{0\}$, we inductively define a family of functions $\{U_m^h\}_{m\in\widetilde{M}} \subset C(\overline{\Omega})$ by the viscosity solutions of 
\begin{equation} \label{AppEq}
\left \{
\begin{array}{rcll}
\dis \frac{1}{\Ga(2-\al)h^{\al}} \left\{ u - \sum_{k=0}^{m-1} C_{m,k}U_k^h \right\} + Lu & = & f(\cdot, mh) & \quad \mbox{in} \,\, \Omega, \\
u & = & 0 & \quad \mbox{on} \,\, \pa\Omega. \\
\end{array}
\right.
\end{equation} 
We notice that, by \cite[Theorem 1]{I}, $U_m^h$ is the distributional solution to \eqref{AppEq} for each $m\in\widetilde{M}\setminus\{0\}$. Moreover 
since $f\in C(\overline\Omega)\subset L^2(\Omega)$, we have $\{U_m^h\}_{m\in\widetilde{M}} \subset H^2(\Omega)\cap C(\overline{\Omega})$ (see \cite{E} for instance).  
Henceforth, by abuse of notation, we write $U_m$ for $U_m^h$.

Based on \cite{GLM}, we define the function $u_c^h:\overline{\Omega} \times [0,T+h] \rightarrow \R$ by
\begin{equation}\label{PC}
u_c^h(x,t) \coloneqq U_m(x) \quad \mbox{for each} \  x \in \overline{\Omega}, \  t \in [mh,(m+1)h), \  m \in \widetilde{M}. 
\end{equation}
Notice that $u_c^h(x, \cdot)$ is clearly piecewise continuous on $[0, T+h]$ for all $ x \in \overline{\Omega}$.
\begin{Th} \label{convisPC}
Let $u_c^h$ be the function given by (\ref{PC}) for any $h>0$. We have $u_{c}^{h} \rightarrow u$ uniformly in $\overline{\Omega} \times [0,T]$ as $h \rightarrow 0$, where $u$ is the unique viscosity solution to (\ref{CP}).
\end{Th}
The proof of Theorem \ref{convisPC} is similar to that of \cite[Theorem 1.1]{GLM},
but we need to take the Dirichlet boundary condition into account.
We give a sketch of the proof here. 

\begin{proof}
We can easily see that $u_{c}^{h}$ is uniformly bounded on $\overline{\Omega} \times [0,T]$, and 
we denote by $\overline{u}, \underline{u}$ the half-relaxed limits of $u_c^h$, that is,  
\begin{equation*}
\begin{array}{lcc}
\overline{u}(x,t) & \coloneqq & \dis \lim_{\de \rightarrow 0} \sup \{u_c^h(y, s) \mid  |x-y|+|t-s| \leq \de, (y, s) \in \overline{\Omega} \times [0,T], 0 < h \leq \de \}, \\
\underline{u}(x,t) & \coloneqq & \dis \lim_{\de \rightarrow 0} \inf \{u_c^h(y, s) \mid  |x-y|+|t-s| \leq \de, (y, s) \in \overline{\Omega} \times [0,T], 0 < h \leq \de \}
\end{array}
\end{equation*}
for all $(x,t) \in \overline{\Omega} \times [0,T]$. 

We prove that $\overline{u}$ is a viscosity subsolution to \eqref{CP} here, and we can similarly prove that  
$\underline{u}$ is a viscosity supersolution to \eqref{CP}. 
We only prove that $\overline{u} \leq 0 $ on $\pa \Omega \times [0, T)$.

Let $w$ be the classical solution to the boundary value problem: 
\begin{equation*}
\left \{
\begin{array}{rcll}
Lw & = & \|f\|_{L^{\infty}} & \quad \mbox{in} \quad \Omega,\\
w & = & 0 & \quad \mbox{on} \quad \pa \Omega.
\end{array}
\right.
\end{equation*} 
By the maximum principle, we have $w(x)>0$ for all $x\in\Omega$. 
Fix $\ep > 0$, and taking $K > 0$ large enough, we have
\begin{equation*}
U_0^h(x) \leq Kw(x) + \ep \quad \mbox{for all} \quad x\in\overline{\Omega} \quad \mbox{and} \quad 0 < h < h_0
\end{equation*}
for some $h_0\in(0,1)$. 
We set $W(x,t) \coloneqq Kw(x)+ \ep$ for any $(x, t) \in \overline{\Omega} \times [0,T]$.
We have 
\begin{align*}
& \dis \frac{1}{\Ga(2-\al)h^{\al}} \left\{ W(\cdot,mh) - \sum_{k=0}^{m-1} C_{m,k}W(\cdot,kh) \right\} + LW(\cdot,mh) - f(\cdot,mh) \\
&=K\|f\|_{L^{\infty}}(1+(mh)^{\al}) - f(\cdot,mh) \geq 0.
\end{align*}

Since $C_{m,k}\ge0$ for all $k=0,1,\ldots,m$, and $m\in\widetilde{M}\setminus\{0\}$, we can easily see that the scheme is monotone by iterating the comparison principle for \eqref{AppEq}, which implies 
\[
U_m\le W(\cdot,mh)\quad\text{on} \ \overline{\Omega} \quad \text{for all} \ m\in\widetilde{M}. 
\]
Thus, we get 
$u_c^h(x, t) \leq W(x, t)$
for all $(x, t) \in \overline{\Omega} \times [0,T]$ and $h > 0$.
Therefore, we obtain
\begin{equation*}
\overline{u}(x, t) \leq W(x, t) = \ep
\end{equation*}
for all $(x, t) \in \pa \Omega \times [0,T)$.
Since $\varepsilon>0$ is arbitrary, we conclude that $\overline{u}\le 0$ on the boundary $\partial\Omega$. 

The rest of the proof is similar to that of \cite[Theorem 1.1]{GLM}, so we omit it. 
\end{proof}


\medskip

We use the following identity, which is well-known (see \cite[Lemma A.1]{KRR}). 
We give the proof for completeness. 

\begin{Lem} \label{RL=Cap}
Let $u \in W^{1, 1}(0, T)$. For all $t \in (0, T)$, we have
\begin{enumerate}
\item[(i)] $ \dis \int_0^t g_{1-\al}(t-s) (u(s)-u(0)) \,ds = \int_0^t g_{2-\al}(t-s) \frac{du}{ds}(s) \,ds$, 
\item[(ii)]$ \dis d_t^\al u(t) = D_t^\al(u-u(0))(t) $.
\end{enumerate}
\end{Lem}

\begin{proof}
Noting that $g_{2-\al}^{\prime}(t) = g_{1-\al}(t)$, we have
\begin{align*}
\dis \int_0^t g_{1-\al}(t-s) (u(s)-u(0)) \,ds &= \int_0^t (-\frac{d}{ds}g_{2-\al}(t-s)) (u(s)-u(0))\,ds \\
&= \dis \int_0^t g_{2-\al}(t-s) \frac{du}{ds}(s) \,ds
\end{align*}
for all $t \in (0, T)$, which proves (i). We also have
\begin{equation*}
 D_t^\al(u-u(0))(t) = \frac{d}{dt}(g_{1-\al} \ast (u - u(0))) = g_{1-\al} \ast \frac{du}{dt} = d_t^\al u(t).
\end{equation*}
\end{proof}

Lemma \ref{RL=Cap} requires that $u$ needs to be absolutely continuous on $[0, T]$.
For this purpose, we modify $u_c^h$ as follows.
Define the function $u^h:\overline{\Omega} \times [0,T+h] \rightarrow \R$ by
\begin{equation}\label{PL}
u^h(x,t) \coloneqq U_m(x) + \dis \frac{U_{m+1}(x) - U_m(x)}{h} (t-mh) \quad
\end{equation}
for all $x\in\overline{\Omega}$, $t \in [mh,(m+1)h)$, and $m\in\widetilde{M}$. 
Clearly, the function $u^h(x,\cdot)$ is absolutely continuous on $[0,T+h]$ for each $x\in\overline{\Omega}$, and  
\begin{equation}\label{conv:uh}
u^h\to u \quad\text{uniformly on} \ \overline{\Omega} \times [0,T] \ \text{as} \ h\to0 
\end{equation}
by Theorem \ref{convisPC}.

Our goal is to prove that $u$ is a distributional solution to \eqref{RLP}.

\begin{Prop} \label{u^h_eq}
Let $u^h$ be the function given by \eqref{PL}. Then, we have
\begin{equation} \label{APE}
D_t^\al(u^h - U_0^h)(t) + Lu^h(t) = f(t) + e^h(t) \quad \mbox{for all} \, \, t \in (0, T+h),
\end{equation}
where we define the error term $e^h : \overline{\Omega} \times [0, T+h) \rightarrow \R$ by
\begin{equation} \label{error_term}
e^h(t) \coloneqq d_t^{\al}u^h(t) - d_t^{\al}u^h(mh) + Lu^h(t) - Lu^h(mh) - (f(t) - f(mh))
\end{equation}
for any $t \in [mh, (m+1)h)$ and $m \in \widetilde{M}$.
\end{Prop}

\begin{proof}
For $h=T/M$ and $m\in\widetilde{M}\setminus\{0\}$, we have, by \eqref{def:dCD}, 
\begin{align*}
d_t^{\al}u^h(mh) &= \dis \int_{0}^{mh} g_{1-\al}(mh-s) \frac{du^h}{ds}(s) \,ds \\
&= \dis \sum_{k=0}^{m-1} \int_{kh}^{(k+1)h} g_{1-\al}(mh-s) \frac{du^h}{ds}(s) \,ds \\
&= \dis \sum_{k=0}^{m-1} \int_{kh}^{(k+1)h} g_{1-\al}(mh-s) \frac{U_{k+1}-U_k}{h} \,ds \\
&= \dis \sum_{k=0}^{m-1} \frac{1}{h^\alpha}\psi(m-k)(U_{k+1}-U_k) \,ds \\
&= \dis \frac{1}{\Ga(2-\al)h^{\al}} \left\{ U_m- \sum_{k=0}^{m-1} C_{m,k}U_k \right\}
= -LU_m + f(mh).
\end{align*}
By Lemma \ref{RL=Cap}, we have $d_t^\al u^h = D_t^\al (u^h - U_0^h)$.
We now observe that \eqref{APE} follows by the definition of \eqref{error_term} of $e^h$.  
\end{proof}

\begin{Rem}
It is worth emphasizing that we cannot expect that viscosity solutions to \eqref{CP} is smooth in general that the Caputo derivative makes sense in the classical sense (see \cite{N} for instance).
Therefore, it is highly unlikely the case that the error term $e^h(t)$ converges to zero as $h \rightarrow 0$ in a strong sense.
In the next section, we prove that $e^h(t)$ converges to zero in a weak sense (see Theorem \ref{errorweak}).
\end{Rem}

\section{The kernel approximation}\label{sec:KA}

In this section we give a key ingredient to prove that $e^h(t)$ goes to zero as $h \rightarrow 0$ in a weak sense.
To estimate $e^h(t)$, the key term
\begin{equation} \label{ess_error}
d_t^\al u^h(t) - d_t^\al u^h(mh)
\end{equation}
needs to be carefully handled.
For this purpose we first find a primitive function for \eqref{ess_error} which will be defined by $G^h[u^h]$ such that 
\begin{equation*}
d_t^\al u^h(mh) = \frac{d}{dt}G^h[u^h](t)
\end{equation*}
for all $t \in (mh, (m+1)h)$ and $m \in \widetilde{M}$.
By Lemma  \ref{RL=Cap}, we have 
\begin{equation*}
d_t^\al u^h(t)
= 
D_t^{\alpha}(u^h-U_0^h)=\frac{d}{dt}g_{1-\alpha}\ast(u^h-U^h_0)
=\frac{d}{dt}G[u^h](t), 
\end{equation*}
where we set 
\[
G[u^h](t) \coloneqq \left(g_{2-\al} \ast \frac{du^h}{ds}\right)(t). 
\]
By definition of $u^h$, we observe that, for $t\in[mh,(m+1)h)$ and $m\in\widetilde{M}\setminus\{0\}$,  
\begin{align*}
\dis & G[u^{h}](t) = \int_0^t g_{2-\al}(t-s) \frac{du^h}{ds}(s) \,ds \\
&= \int_{mh}^{t} g_{2-\al}(t-s) \frac{U_{m+1}-U_m}{h} \,ds + \sum_{k=0}^{m-1} \int_{kh}^{(k+1)h} g_{2-\al}(t-s) \frac{U_{k+1}-U_k}{h} \,ds \\
&= g_{3-\al}(t-mh) \frac{U_{m+1}-U_m}{h} + \sum_{k=0}^{m-1} \{ (g_{3-\al}(t-kh) - g_{3-\al}(t-(k+1)h \} \frac{U_{k+1}-U_k}{h}.
\end{align*}
We approximate $G[u^h]$ by approximating $g_{3-\alpha}$. 
We set 
\begin{align}
\dis G^h[u^h](t) & \coloneqq g_{3-\al}^h(t-mh) \frac{U_{m+1}-U_m}{h} \nonumber \\
&+ \sum_{k=0}^{m-1} \{ g_{3-\al}^h(t-kh) - g_{3-\al}^h(t-(k+1)h) \} \frac{U_{k+1}-U_k}{h}
\label{app:kernel}
\end{align}
for all $t\in[mh,(m+1)h)$ and $m\in\widetilde{M}$. 
Here, we choose a family of functions $\{g_{3-\al}^h\}_{h>0}$ on $[0,T+h]$ satisfying, for all $h>0$,  
\begin{enumerate}[{(}a{)}]
\item $g_{3-\al}^h$ is continuous and linear on $(mh,(m+1)h)$ for $m\in\widetilde{M}$, 
\item $\dis \frac{d}{dt} G^h[u^h](t) = d_t^\al u(mh) $ for all $mh < t < (m+1)h$ and $m \in \widetilde{M}$, 
\item $\dis \sup_{t\in(0,T)} |(g_{3-\al} - g_{3-\al}^h)(t)| \leq C_T h$ for some $C_T \geq 0$.
\end{enumerate}

By (b), we have 
\begin{equation} \label{eq:G}
d_t^\al u^h(t) - d_t^\al u^h(mh) = \frac{d}{dt} \left( G[u^h] - G^h[u^h] \right)(t)
\end{equation}
for all $t \in (mh, (m+1)h)$ and $m \in \widetilde{M}$.
Such  $\{g_{3-\alpha}^h\}_{h>0}$ actually exists as follows. 
\begin{Lem} \label{kerap}
Set
\begin{equation*}
g_{3-\al}^h(t) \coloneqq 
\left\{
\begin{array}{ll}
0 \quad &\mbox{for} \quad 0 \leq t \leq h, \\
g_{3-\al}^{\prime}(mh)(t-mh) + \dis \sum_{k=0}^{m-1} g_{3-\al}^{\prime}(kh)h \, & \mbox{for} \quad mh \leq t \leq (m+1)h, m \in \widetilde{M}\setminus\{0\}.
\end{array}
\right. 
\end{equation*}
Then, (a)--(c) hold.
\end{Lem}

\begin{proof}
We can easily check (a) since
\begin{equation*} 
(g_{3-\al}^h)^{\prime}(t) = g_{3-\al}^{\prime}(mh) = g_{2-\al}(mh) \quad \mbox{for} \quad mh < t < (m+1)h 
\ \text{and} \ m \in \widetilde{M}.
\end{equation*}
By a direct computation,
\begin{align*}
\dis \frac{d}{dt}G^h[u^h](t) &= (g_{3-\al}^h)^{\prime}(t-mh) \frac{U_{m+1}-U_m}{h} \\
&+ \sum_{k=0}^{m-1} \{ (g_{3-\al}^h)^{\prime}(t-kh) - (g_{3-\al}^h)^{\prime}(t-(k+1)h) \} \frac{U_{k+1}-U_k}{h} \\
&= \sum_{k=0}^{m-1} \{ g_{2-\al}((m-k)h) - g_{2-\al}((m-k-1)h) \} \frac{U_{k+1}-U_k}{h} 
= d_t^{\al}u^h(mh)
\end{align*}
for all $t \in (mh, (m+1)h)$ and $m \in \widetilde{M}$. 
This proves (b).

Noting that
\begin{align*}
\dis \frac{d}{dt}(g_{3-\al} - g_{3-\al}^h) &= g_{3-\al}^{\prime}(t) - g_{3-\al}^{\prime}(mh) = g_{2-\al}(t) - g_{2-\al}(mh) \geq 0
\end{align*}
for all $t \in [mh, (m+1)h)$ and $m \in \widetilde{M}$, $g_{3-\al} - g_{3-\al}^h$ is a nondecreasing function.
Thus, we can easily check that
\begin{equation*}
\dis \sup_{t \in (0,T)} |(g_{3-\al} - g_{3-\al}^h)(t)| = |g_{3-\al}(T) - g_{3-\al}^h(T)|.
\end{equation*}
Noting that 
\[
g_{3-\alpha}^h(T)=g_{3-\alpha}^h(Mh)
=\sum_{k=0}^{M-1}g_{3-\alpha}'(kh)h
=\sum_{k=0}^{M-1}g_{2-\alpha}(kh)h, 
\]
we have
\begin{align*}
|g_{3-\al}(T) - g_{3-\al}^h(T)|
&= \dis \frac{1}{\Ga (3-\al)} \left( (Mh)^{2-\al} - (2-\al)\sum_{k=0}^{M-1}(kh)^{1-\al}h  \right) \\
&= \dis \frac{h^{2-\al}}{\Ga (3-\al)} \left( M^{2-\al} - (2-\al)\sum_{k=0}^{M-1}k^{1-\al} \right).
\end{align*}
Note
\begin{equation*}
\dis (2-\al) \sum_{k=0}^{M-1} k^{1-\al} \geq (2-\al) \int_{0}^{M-1} t^{1-\al} \,dt = (M-1)^{2-\al}.
\end{equation*}
We have
\begin{equation*}
\dis |g_{3-\al}(T) - g_{3-\al}^h(T)|
\leq \frac{T^{2-\al}}{\Ga (3-\al)} \left( 1- (1-\frac{1}{M})^{2-\al} \right)
\leq C T^{2-\al} \frac{1}{M} = C T^{1-\al}h.
\end{equation*}
\end{proof}

Finally we give an important estimate of $G[u^h] - G^h[u^h]$.
We recall that we write $U_k$ for $U_k^h$ for simplicity.
\begin{Lem} \label{Tonari}
We fix $T=Mh$. Let $U_k$ be the functions given by \eqref{ini_app}, \eqref{AppEq} for $k\in\widetilde{M}$.
For all $\ep > 0$, there exists $h_0 > 0 $ such that, for all $0 < h < h_0$,
\begin{equation*}
\|U_{k+1} - U_k\|_{\infty} < \ep \quad \mbox{for all} \quad k \in \widetilde{M}. 
\end{equation*}
\end{Lem}
This is a straightforward result of the uniform convergence of $u_c^h$ by Theorem \ref{convisPC}.
Next, we prove a key ingredient of the paper to prove that the error term $e^h$ goes to zero in a weak sense as $h \rightarrow 0$.
\begin{Lem} \label{errorweak}
Let $G^h[u^h]$ be the function  given by \eqref{app:kernel}. 
For all $\ep > 0$, there exists $h_0 > 0$ such that, for all $0 < h < h_0$, 
\begin{equation}\label{ineq:G-Gh}
\left| \int_0^T \int_{\Omega} (G[u^h] - G^h[u^h]) \eta \ph\, dxdt \right| \leq C_T \|\ph\|_{L^1} (\|\eta\|_{\infty} + \|\eta^{\prime}\|_{\infty}) \ep,
\end{equation}
for all $\ph \in C_c^{\infty}(\Omega)$ , $\eta \in C_c^{\infty}(0,T)$, 
where $C_T$ is a positive constant which is independent of $\ep$.
\end{Lem}

\begin{proof}
Fix $\ep > 0$. Set $l^h(t) \coloneqq (g_{3-\al} - g_{3-\al}^h)(t)$.
Then we have
\begin{align*}
\dis (G[u^h]-G^h[u^h])(t)
&= l^h(t-mh) \frac{U_{m+1}-U_m}{h} \\
&+ \sum_{k=0}^{m-1} \{ l^h(t-kh) - l^h(t-(k+1)h) \} \frac{U_{k+1}-U_k}{h}.
\end{align*}
We multiply the above equation by $\ph \in C_c^{\infty}(\Omega)$ and $\eta \in C_c^{\infty}(0,T)$.
Take a small $h > 0$ satisfying $\eta(t) = 0$ for all $t \in [0, h] \cup [T-h, T].$
Hence,
\begin{align*}
\dis \int_0^T\int_{\Omega} (G[u^h] - &G^h[u^h]) \eta(t) \ph(x) \,dx dt
=  \int_h^T\int_{\Omega} l^h(t-mh) \frac{U_{m+1}-U_m}{h} \eta(t) \ph(x) \,dx dt \\
&+ \int_h^T\int_{\Omega}  \sum_{k=0}^{m-1} \{ l^h(t-kh) - l^h(t-(k+1)h) \} \frac{U_{k+1}-U_k}{h} \eta(t) \ph(x) \,dx dt \\
&\revcoloneqq I_1 + I_2.
\end{align*}

We have
\begin{align*}
I_1
&=  \frac{1}{h} \sum_{m=1}^{M-1} \int_{mh}^{(m+1)h} l^h(t-mh) \eta(t) \,dt \int_{\Omega} (U_{m+1}(x)-U_m(x)) \ph(x) \,dx \\
&= \frac{1}{h} \sum_{m=1}^{M-1} \int_{0}^{h} l^h(t) \eta(t+mh) \,dt \int_{\Omega} (U_{m+1}(x)-U_m(x)) \ph(x) \,dx \\
&\leq \frac{1}{h} \sum_{m=1}^{M-1} \int_{0}^{h} l^h(t) \|\eta\|_{L^{\infty}} \|U_{m+1}-U_m\|_{L^{\infty}} \|\ph\|_{L^1}\,dt.
\end{align*}
Note that by property (c) of $g_{3-\alpha}^h$ we have 
\[
\int_0^hl^h(t)\,dt=\int_0^h(g_{3-\alpha}-g_{3-\alpha}^h)(t)\,dt\le C_Th^2. 
\]
By Lemma \ref{Tonari}, we have 
\begin{align*}
&|I_1| \leq \frac{1}{h} \sum_{m=1}^{M-1} \int_{0}^{h} l^h(t) \|\eta\|_{L^{\infty}} \|U_{m+1}-U_m\|_{L^{\infty}} \|\ph\|_{L^1}\,dt\\
\leq &
C_T\|\eta\|_{L^\infty}\|\varphi\|_{L^1}Mh\ep=C_T T\|\eta\|_{L^\infty}\|\varphi\|_{L^1}\ep.
\end{align*}

Next, we have
\begin{equation*}
I_2 = \int_{\Omega} \sum_{m=1}^{M-1} \int_{mh}^{(m+1)h} \sum_{k=0}^{m-1} \{ l^h(t-kh) - l^h(t-(k+1)h) \} \frac{U_{k+1}(x)-U_{k}(x)}{h} \eta(t) \ph(x) \,dtdx.
\end{equation*}
Note that by some tedious computations we have
\begin{align*}
I_2
&= \sum_{m=1}^{M-1} \sum_{k=0}^{m-1} \int_{mh}^{(m+1)h} (l^h(t-kh) - l^h(t-(k+1)h)) \eta(t)\,dt  \int_{\Omega} \frac{U_{k+1}(x)-U_{k}(x)}{h} \ph(x) \,dx \\
&= \sum_{m=1}^{M-1} \int_{mh}^{T} (l^h(t-(m-1)h) - l^h(t-mh)) \eta(t)\,dt  \int_{\Omega} \frac{U_{m}(x)-U_{m-1}(x)}{h} \ph(x) \,dx \\
&= \frac{1}{h} \sum_{m=1}^{M-1} \int_{mh}^{T} l^h(t-(m-1)h) \eta(t)\,dt  \int_{\Omega} (U_{m}(x)-U_{m-1}(x)) \ph(x) \,dx\\
& \hspace*{2cm} - \frac{1}{h} \sum_{m=1}^{M-1} \int_{mh}^{T} l^h(t-mh) \eta(t) \,dt \int_{\Omega} (U_{m}(x)-U_{m-1}(x)) \ph(x) \,dx.
\end{align*}
We extend $\eta(t) = 0$ for $t < 0$ and $t > T$ and keep to use the same notation by abuse of notation.
Moreover, we have
\begin{align*}
I_2
&= \frac{1}{h} \sum_{m=1}^{M-1} \left \{ \int_{h}^{T-(m-1)h} l^h(t) \eta(t+(m-1)h) \,dt \right. \\
& \hspace*{3cm} - \left. \int_{0}^{T-mh} l^h(t) \eta(t+mh) \,dt \right \} \int_{\Omega} (U_{m}(x)-U_{m-1}(x)) \ph(x) \,dx \\
&\leq \frac{1}{h} \sum_{m=1}^{M-1} \int_{0}^{h} l^h(t) \eta(t+(m-1)h) \,dt \int_{\Omega} (U_{m}(x)-U_{m-1}(x)) \ph(x) \,dx \\
& \hspace*{3cm} + \sum_{m=1}^{M-1} \int_{0}^{T} l^h(t) \frac{\eta(t+(m-1)h) - \eta(t+mh)}{h} \,dt \int_{\Omega} (U_{m}(x)-U_{m-1}(x)) \ph(x) \,dx\\
&\leq \frac{1}{h} \sum_{m=1}^{M-1} \int_{0}^{h} l^h(t) \|\eta\|_{L^{\infty}} \|U_{m}-U_{m-1}\|_{L^{\infty}} \|\ph\|_{L^1} \,dt  + \sum_{m=1}^{M-1} \int_{0}^{T} l^h(t) \|\eta^{\prime}\|_{L^{\infty}} \|U_{m}-U_{m-1}\|_{L^{\infty}} \|\ph\|_{L^1}\,dt.
\end{align*}
By a similar argument to the above, we get 
\begin{align*}
|I_2|
&\leq C_T(\|\eta\|_{L^\infty}+\|\eta'\|_{L^\infty})\|\varphi\|_{L^1}h \sum_{m=1}^{M-1} \|U_{m}-U_{m-1}\|_{L^\infty} \\
&\leq C_T(\|\eta\|_{L^\infty}+\|\eta'\|_{L^\infty})\|\varphi\|_{L^1}Mh\ep = C_T(\|\eta\|_{L^\infty}+\|\eta'\|_{L^\infty})\|\varphi\|_{L^1}T\ep.
\end{align*}
Consequently, we obtain \eqref{ineq:G-Gh}. 
\end{proof}

\section{Energy estimate}\label{sec:EE}

\begin{Lem} \label{dtal_coer}
Let $\{ U_k \}$ be a family of functions in $L^2(\Omega)$. 
Then, we have
\begin{equation*}
\dis \left( \frac{1}{\Ga(2-\al)h^{\al}} \left\{ U_m - \sum_{k=0}^{m-1} C_{m,k}U_k \right\}, U_m \right)_{L^2}
\geq
\frac{1}{2\Ga(2-\al)h^{\al}} \left\{ \|U_m\|_{L^2}^2 - \sum_{k=0}^{m-1} C_{m,k}\|U_k\|_{L^2}^2 \right\}, 
\end{equation*}
where we denote $L^2$-norm by $(\cdot,\cdot)_{L^2}$. 
\end{Lem}

\begin{proof}
We have
\begin{align*}
\dis & \left( \frac{1}{\Ga(2-\al)h^{\al}} \left\{ U_m - \sum_{k=0}^{m-1} C_{m,k}U_k \right\}, U_m \right)_{L^2} \\
& = 
\dis \frac{1}{\Ga(2-\al)h^{\al}} \left\{ \|U_m\|_{L^2}^2 - \sum_{k=0}^{m-1} C_{m,k}(U_k, U_m)_{L^2} \right\}.
\end{align*}
Noting that
\begin{equation*}
\sum_{k=0}^{m-1} C_{m,k} = 1 \quad \mbox{for all} \quad m \in \widetilde{M}\setminus\{0\},
\end{equation*}
by the Schwarz inequality, we obtain
\begin{equation*}
\sum_{k=0}^{m-1} C_{m,k} (U_k, U_m)_{L^2} \leq \frac{1}{2} \left( \|U_m\|_{L^2}^2 + \sum_{k=0}^{m-1} C_{m,k} \|U_k\|_{L^2}^2 \right),
\end{equation*}
which finishes the proof.
\end{proof}

\begin{Rem}
We notice that the energy estimate for the Caputo fractional derivative 
\begin{equation*}
(d_t^\al u(t), u(t))_{L^2} \geq \dis \frac{1}{2} d_t^\al(\|u\|_{L^2}^2)(t) \quad \mbox{for} ~~ u \in W^{1,1}(0, T; L^2(\Omega)).  
\end{equation*}
is known. 
We refer to \cite[Theorem 3.3]{KRY} for the proof. 
We can regard Lemma \ref{dtal_coer} as a discrete analog of this estimate.  
\end{Rem}

\begin{Lem} \label{int_dtal}
Let $\{ U_k \}$ be a family of functions in $L^2(\Omega)$. 
Then, we have
\begin{align*}
& \sum_{n=1}^m \frac{1}{\Ga(2-\al)h^{\al}} \left\{ \|U_n\|_{L^2}^2 - \sum_{k=0}^{n-1} C_{n,k}\|U_k\|_{L^2}^2 \right\} \\
& \quad \quad \quad \geq \frac{(mh)^{-\al}}{\Ga(1-\al)} \sum_{k=1}^m \|U_k\|_{L^2}^2 - \frac{(mh)^{1-\al}}{\Ga(2-\al)h} \|U_0\|_{L^2}^2 \quad \mbox{for all} \quad m \in \widetilde{M}\setminus\{0\}.
\end{align*}
\end{Lem}

\begin{proof}
Note that for any $m \in \widetilde{M}$, we have 
\begin{align*}
& \sum_{n=1}^m \left\{ \|U_n\|_{L^2}^2 - \sum_{k=0}^{n-1} C_{n,k}\|U_k\|_{L^2}^2 \right\} \\
=&\,
\|U_m\|_{L^2}^2+ \sum_{k=1}^{m-1} \left(1 - \sum_{n=k+1}^{m} C_{n,k} \right) \|U_k\|_{L^2}^2 - \sum_{n=1}^{m} C_{n, 0} \|U_0\|_{L^2}^2.
\end{align*}
Moreover,
\begin{align*}
1 - \sum_{n=k+1}^{m} C_{n,k} &= 1 - \Ga(2-\al)(\psi(1) - \psi(m+1-k)) \\
&= \Ga(2-\al)\psi(m+1-k) 
= (m+1-k)^{1-\al} - (m-k)^{1-\al}.
\end{align*}
Due to the concavity of $r \mapsto r^{1-\al}$, we have, for all $  k= 1, 2, \cdots, m$, 
\begin{equation*}
(m+1-k)^{1-\al} - (m-k)^{1-\al} \geq m^{1-\al} - (m-1)^{1-\al} \geq (1-\al)m^{-\al}.
\end{equation*}
Also, note that
\begin{equation*}
\sum_{n=1}^{m} C_{n, 0} = m^{1-\al}.
\end{equation*}
Therefore,
\begin{align*}
& \sum_{n=1}^m \frac{1}{\Ga(2-\al)h^{\al}} \left\{ \|U_n\|_{L^2}^2 - \sum_{k=0}^{n-1} C_{n,k}\|U_k\|_{L^2}^2 \right\} \\
&= 
\frac{1}{\Ga(2-\al)h^{\al}} \left\{\|U_m\|_{L^2}^2+ \sum_{k=1}^{m-1} \left(1 - \sum_{n=k+1}^{m} C_{n,k} \right) \|U_k\|_{L^2}^2 \right\}
- \frac{1}{\Ga(2-\al)h^{\al}} \sum_{n=1}^{m} C_{n, 0} \|U_0\|_{L^2}^2 \\
&\geq \frac{(mh)^{-\al}}{\Ga(1-\al)} \sum_{k=1}^{m} \|U_k\|_{L^2}^2 
-\frac{(mh)^{1-\al}}{\Ga(2-\al)h} \|U_0\|_{L^2}^2.
\end{align*}
\end{proof}

\begin{Rem}
We note here that our result is a discrete analog of 
\begin{equation*}
\dis \int_0^t d_s^\al (\|u\|_{L^2}^2)(s) \,ds \geq \frac{t^{-\al}}{\Ga(1-\al)} \int_0^t \|u(s)\|_{L^2}^2 \,ds - \frac{t^{1-\al}}{\Ga(2-\al)}\|u(0)\|_{L^2}^2
\end{equation*}
 for all $u \in W^{1,1}(0, T; L^2(\Omega))$ 
 in \cite[Theorem 3.3]{KRY}. 
\end{Rem}

\begin{Th} \label{dis_energy_estimate}
Let $\{ U_k \}_{k \in \widetilde{M}}$ be the family of the functions given by (\ref{ini_app}) and (\ref{AppEq}). 
Then, we have
\begin{equation*}
h \left( \sum_{m=1}^M \|U_m\|_{L^2}^2 + \sum_{m=1}^M \|\na U_m\|_{L^2}^2 \right)
\leq C_T \left( \|U_0\|_{L^2}^2 + h \sum_{m=1}^M \|f_m\|_{L^2}^2 \right),
\end{equation*}
where we set $f_m(x) \coloneqq f(x, mh)$ for $x \in \Omega$ and $m \in \widetilde{M}$.
\end{Th}

\begin{proof}
Fix $m \in \widetilde{M}\setminus\{0\}$. We have
\begin{equation} \label{diseq}
\dis \frac{1}{\Ga(2-\al)h^{\al}} \left\{ U_m - \sum_{k=0}^{m-1} C_{m,k} U_k \right\} + LU_m = f_m.
\end{equation}
Multiplying (\ref{diseq}) by $U_m$ and integrating on $\Omega$, by Lemma \ref{dtal_coer} and the uniform ellipticity of $L$, we get
\begin{align*}
&\dis \frac{1}{2\Ga(2-\al)h^{\al}} \left\{ \|U_m\|_{L^2}^2 - \sum_{k=0}^{m-1} C_{m,k}\|U_k\|_{L^2}^2 \right\}
+ \lambda \|\na U_m\|_{L^2}^2\\
\leq& \,
(f_m, U_m)_{L^2}
\le
\frac{1}{2\ep} \|f_m\|^2_{L^2} + \frac{\ep}{2}\|U_m\|^2_{L^2}
\end{align*}
for any $\ep > 0$.
Summing up on $m$, by Lemma \ref{int_dtal}, we arrive at 
\begin{align*}
& \frac{h}{\ep} \sum_{m=1}^{M} \|f_m\|^2_{L^2} + \ep h \sum_{m=1}^{M} \|U_m\|^2_{L^2} \\
& \geq \dis \frac{h}{\Ga(2-\al)h^{\al}} \sum_{m=1}^{M} \left\{ \|U_m\|_{L^2}^2 - \sum_{k=0}^{m-1} C_{m,k}\|U_k\|_{L^2}^2 \right\} + 2 \lambda h \sum_{m=1}^{M} \|\na U_m\|_{L^2}^2 \\
& \geq \frac{(Mh)^{-\al}}{\Ga(1-\al)} h \sum_{m=1}^{M} \|U_m\|_{L^2}^2 - \frac{(Mh)^{1-\al}}{\Ga(2-\al)} \|U_0\|_{L^2}^2 + 2 \lambda h \sum_{m=1}^{M} \|\na U_m\|_{L^2}^2,
\end{align*}
which implies
\begin{equation*}
\left( \frac{T^{-\al}}{\Ga(1-\al)} - \ep \right) h \sum_{m=1}^M \|U_m\|_{L^2}^2 + 2 \lambda h \sum_{m=1}^M \|\na U_m\|_{L^2}^2 \leq \frac{T^{1-\al}}{\Ga(2-\al)} \|U_0\|_{L^2}^2 + \frac{1}{\ep}h \sum_{k=1}^M \|f_m\|_{L^2}^2.
\end{equation*}
Take $\dis \ep \coloneqq \frac{(1-\lambda)T^{-\al}}{\Ga(1-\al)}$ to get 
\begin{equation*}
h \left( \sum_{m=1}^M \|U_m\|_{L^2}^2 + \sum_{m=1}^M \|\na U_m\|_{L^2}^2 \right)
\leq \frac{C_T}{\lambda} \left( \|U_0\|_{L^2}^2 + h \sum_{m=1}^M \|f_m\|_{L^2}^2 \right)
\end{equation*}
for some $C_T \geq 0$ independent of $h$. 
\end{proof}

\begin{Th} \label{energy_estimate}
Let $u_c^h$ and $u^h$ be the functions given by (\ref{PC}) and (\ref{PL}), respectively.
For any $\ep > 0$, there exists $h_0 > 0$ such that for all $h \in (0, h_0)$, we have
\begin{equation*}
\|u^h_c\|_{L^2(0, T; H^1(\Omega))} + \|u^h\|_{L^2(0, T; H^1(\Omega))}
\leq C_T ( \|u_0\|_{L^2} + \|f\|_{L^2(0, T; L^2(\Omega))} + \ep )
\end{equation*}
for some $C_T \geq 0$ which is independent of $\ep$ and $h$.
\end{Th}

\begin{proof}
Fix $\ep > 0$.
Noting that $u_c^h(t) = U_m$ for $mh \leq t < (m+1)h$, we have, by Theorem \ref{dis_energy_estimate},
\begin{align*}
\int_h^T \|u_c^h(t)\|_{H^1}^2 \,dt &= \sum_{m=1}^{M-1} \int_{mh}^{(m+1)h} \|U_m\|_{H^1}^2 \,dt 
= h \sum_{m=1}^{M-1} \|U_m\|_{H^1}^2 \\
& \leq C_T \left( \|U_0\|_{L^2}^2 + h \sum_{m=1}^M \|f_m\|_{L^2}^2 \right).
\end{align*}
Take $h_0 > 0$ small so that for $0 < h < h_0$,
\begin{align*}
\|U_0\|_{L^2}^2 &\leq \|u_0\|_{L^2}^2 + \frac{\ep}{2}, \quad
h \sum_{m=1}^M \|f_m\|_{L^2}^2 \leq \|f\|_{L^2(0, T; L^2(\Omega))}^2 + \frac{\ep}{2}.
\end{align*}
Thus, for all $h \in (0, h_0)$, we get
\begin{equation*}
\|u^h_c\|_{L^2(0, T; H^1(\Omega))}^2 \leq C_T ( \|u_0\|_{L^2}^2 + \|f\|_{L^2(0, T; L^2(\Omega))}^2 + \ep ).
\end{equation*}

Next, we give an estimate for $\|u^h\|_{L^2(0, T; H^1(\Omega))}$.
We have
\begin{align*}
\int_h^T \|u^h(t)\|_{H^1}^2 dt 
&=
\sum_{m=1}^{M-1} \int_{mh}^{(m+1)h} \left\|U_m + \frac{U_{m+1} - U_m}{h}(t - mh) \right\|_{H^1}^2 dt \\
& \leq \sum_{m=1}^{M-1} \int_{mh}^{(m+1)h} \left( \|U_m\|_{H^1} + \frac{\|U_{m+1} - U_m\|_{H^1}}{h}(t - mh) \right)^2 dt \\
& = \sum_{m=1}^{M-1} \frac{h}{3 \|U_{m+1} - U_m \|} \left\{ (\|U_m\|_{H^1} + \|U_{m+1} - U_m\|_{H^1})^3 - \|U_m\|_{H^1}^3 \right\} \\
& = \frac{h}{3} \sum_{m=1}^{M-1} \Big\{ (\|U_m\|_{H^1} + \|U_{m+1} - U_m\|_{H^1})^2  \\
& \quad \quad \quad \quad + (\|U_m\|_{H^1} + \|U_{m+1} - U_m\|_{H^1})\|U_m\|_{H^1} + \|U_m\|_{H^1}^2 \Big\} \\
& \leq C h \sum_{m=1}^{M} \|U_m\|_{H^1}^2
\end{align*}
for some $C \geq 0$.
By the same argument as above we get the conclusion.
\end{proof}

\section{Proof of Theorem \ref{Main_Th}}\label{sec:main-proof}
\begin{Lem} \label{error_weakconv}
Let $e^h$ be the function defined by (\ref{error_term}).
We have
\begin{equation*}
\int_0^T \int_{\Omega} e^h(x, t) \ph(x) \eta(t) dx dt \rightarrow 0 \quad \mbox{as} \, h \rightarrow 0
\end{equation*}
for all $\ph \in C_c^{\infty}(\Omega)$ and $\eta \in C_c^{\infty}(0, T)$.
\end{Lem}

\begin{proof}
Let $\ph \in C_c^{\infty}(\Omega)$ and $\eta \in C_c^{\infty}(0, T)$.
We have
\begin{align*}
& \left|\int_0^T \int_{\Omega} e^h(x, t) \ph(x) \eta(t) \,dx dt\right| \\
& \le 
\left|\int_0^T \int_{\Omega} \left( d_t^{\al}u^h(t) - d_t^{\al}u^h(mh) \right) \ph(x) \eta(t) \,dx dt\right| \\
& +\left| \int_0^T \int_{\Omega} \left( \partial_{x_i}(a_{i, j}(x)\partial_{x_j}u^h(t)) - \partial_{x_i}(a_{ij}(x)\partial_{x_j}u^h(mh)) \right) \ph(x) \eta(t) \,dx dt\right| \\
& +\left| \int_0^T \int_{\Omega} \left( f(t) - f(mh) \right) \ph(x) \eta(t) \,dx dt\right| \revcoloneqq I_1 + I_2 + I_3.
\end{align*}
By (\ref{eq:G}) and Lemma \ref{errorweak}, for any $\ep > 0$, there exists $h_0 > 0$ such that
\begin{equation*}
I_1 =\left| - \int_0^T \int_{\Omega} \left( G[u^h](t) - G^h[u^h](t) \right) \ph(x) \eta^{\prime}(t) \,dx dt\right| 
\leq C_T \|\ph\|_{L^1} (\|\eta^{\prime}\|_{\infty} + \|\eta^{\prime\prime}\|_{\infty}) \ep
\end{equation*}
for all $h \in (0, h_0)$.
In light of Theorem \ref{energy_estimate} and the weak compactness, there exists a subsequence 
$\{u^{h_n} \}_{n \in \N}$ such that 
\begin{equation*}
u^{h_n} \rightharpoonup v \quad \mbox{in} \quad L^2(0,T;H_0^1(\Omega)) 
\quad{as} \ n\to\infty 
\end{equation*}
for $v \in L^2(0,T;H_0^1(\Omega))$.
Noting that $u^h \rightarrow u$ in $C(\overline{\Omega} \times [0, T])$ as $h \rightarrow 0$, 
by Theorem \ref{convisPC} and \eqref{conv:uh}
we have $u = v$ on $\overline{\Omega} \times [0, T]$.
Therefore,
\begin{equation*}
I_2 = \left|-\int_0^T \int_{\Omega} a_{ij}(x) \left( \partial_{x_j} u^{h_n}(x, t) - \partial_{x_j} u^{h_n}(x, mh) \right) \partial_{x_i} \ph(x) \eta(t) \,dx dt\right| \rightarrow 0
\end{equation*}
as $n \rightarrow +\infty$. 
Also, $I_3 \rightarrow 0$ as $h \rightarrow 0$.

Therefore,
\begin{equation*}
\varlimsup_{h \to 0} \left|\int_0^T \int_{\Omega} e^{h}(x, t) \ph(x) \eta(t) \,dx dt\right| \leq C_T \|\ph\|_{L^1} (\|\eta^{\prime}\|_{\infty} + \|\eta^{\prime\prime}\|_{\infty}) \ep
\end{equation*}
for all $\ep > 0$, which implies the conclusion.
\end{proof}

\begin{proof}[Proof of Theorem {\rm\ref{Main_Th}}]
We first assume that $u$ is the unique viscosity solution to (\ref{CP}).
Let $u_c^h$ and $u^h$ be given by (\ref{PC}) and (\ref{PL}) respectively.
By Theorem \ref{convisPC} and \eqref{conv:uh}, we have
\begin{equation*}
u_c^h, u^h \rightarrow u \quad \mbox{uniformly on} \quad \overline{\Omega} \times [0,T] \quad \mbox{as} \quad h \rightarrow 0.
\end{equation*}
We prove that $u$ is the distributional solution to (\ref{RLP}), i.e.,
\begin{equation*}
u \in L^2(0,T;H_0^1(\Omega)), \, \,  g_{1-\al} \ast (u-u_0) \in {}_{0}H^1(0,T;H^{-1}(\Omega)) 
\end{equation*}
and $u$ satisfies weak form \eqref{wf}.

By the same argument as in the proof of Lemma \ref{error_weakconv}, we have
\begin{equation*}
u_c^{h_n} \rightharpoonup u \quad \mbox{in} \quad L^2(0,T;H_0^1(\Omega)) \quad \mbox{as} \quad n \rightarrow +\infty
\end{equation*}

Next, we prove that $u$ satisfies weak form of \eqref{wf}. 
By Proposition \ref{u^h_eq}, we have
\begin{equation} \label{AE}
D_t^{\al}(u^h-U_0^h)(t) + L u^h(t) = f(t) + e^h(t) \quad \mbox{for} \quad (mh, (m+1)h) \quad \mbox{and} \quad m \in \widetilde{M}.
\end{equation}
Multiplying (\ref{AE}) by $\ph \in C_c^{\infty}(\Omega)$ and $\eta \in C_c^{\infty}(0,T)$, and integrating on $\Omega \times (0,T)$, we get
\begin{equation*}
\dis \int_0^T \int_{\Omega} \frac{\pa}{\pa t}(g_{1-\al} \ast (u^h-U_0^h)) \eta \ph \,dx dt
+ \int_0^T \int_{\Omega} L u^h \cdot \eta \ph \,dx dt
= \dis \int_0^T \int_{\Omega} (f + e^h) \eta \ph \,dx dt.
\end{equation*}
Integrating by parts, we get
\begin{align*}
\dis \int_0^T \int_{\Omega} \frac{\pa}{\pa t}(g_{1-\al} \ast (u^h-U_0^h)) & \eta \ph \,dx dt
= \dis - \int_0^T \int_{\Omega} (g_{1-\al} \ast (u^h-U_0^h)) \eta^{\prime} \ph \,dx dt \\
&\rightarrow \dis - \int_0^T \int_{\Omega} (g_{1-\al} \ast (u-u_0)) \eta^{\prime} \ph \,dx dt \quad \mbox{as} \, h \rightarrow 0,
\end{align*}
since we have $u^h \rightarrow u$ uniformly on $\overline{\Omega} \times [0, T]$.
Also, noting that $u^h \rightharpoonup u$ in $L^2(0,T;H_0^1(\Omega))$ as $h \rightarrow 0$, we have
\begin{align*}
- \int_0^T \int_{\Omega} \partial_{x_i}(a_{ij}(x) \partial_{x_j} u^h) \cdot \eta \ph \,dx dt
&= \int_0^T \int_{\Omega} a_{ij}(x) \partial_{x_j} u^h \partial_{x_i} \ph \cdot \eta \,dx dt \\
&\rightarrow \int_0^T \int_{\Omega} a_{ij}(x) \partial_{x_j} u \partial_{x_i} \ph \cdot \eta \,dx dt
\quad \mbox{as} \, h \rightarrow 0. 
\end{align*}
Here, we use the Einstein summation convention. By Lemma \ref{error_weakconv}, we have
\begin{equation*}
\dis \int_0^T \int_{\Omega} e^h \eta \ph \,dx dt \quad \mbox{as} \, h \rightarrow 0.
\end{equation*}
Therefore, we obtain
\begin{equation*} 
\dis - \int_0^T \eta^{\prime} \int_{\Omega} (g_{1-\al} \ast (u-u_0)) \ph \,dx dt
+ \int_0^T \eta \int_{\Omega} a_{ij}(x) \partial_{x_j} u \partial_{x_i} \ph \,dx dt
= \int_0^T \eta \int_{\Omega} f \ph \,dx dt, 
\end{equation*}
which implies \eqref{wf}. 

Finally, we prove that $g_{1-\al} \ast (u-u_0) \in {}_{0}H^1(0,T;H^{-1}(\Omega))$.
By Young's inequality, we have
\begin{equation*}
\|g_{1-\al} \ast (u-u_0)\|_{L^2(0,T;H^{-1}(\Omega))}
\leq \|g_{1-\al}\|_{L^1(0,T)} \|u-u_0\|_{L^2(0,T;H^{-1}(\Omega))}
< \infty.
\end{equation*}
Also, by \eqref{wf}, we have
\begin{align*}
\dis \left\|\frac{d}{dt} [g_{1-\al} \ast (u-u_0)](t) \right\|_{H^{-1}(\Omega)}
&= \sup_{\|v\|_{H_0^1(\Omega)} = 1} \left| \int_{\Omega} \frac{d}{dt} [g_{1-\al} \ast (u-u_0)](t) v \,dx \right| \\
&= \sup_{\|v\|_{H_0^1(\Omega)} = 1} \left| - \int_{\Omega} a_{ij}(\cdot) \partial_{x_j} u(\cdot, t) \partial_{x_i} v \,dx + \int_{\Omega} fv \,dx \right| \\
&\leq \|a_{ij}\|_{L^{\infty}(\Omega)} \|u(t)\|_{H^1({\Omega})} + \|f(t)\|_{L^2({\Omega})}<\infty 
\end{align*}
for almost every  $t \in (0,T)$.  
Thus, we get $ \frac{d}{dt} [g_{1-\al} \ast (u-u_0)] \in L^2(0,T; H^{-1}(\Omega)) $.
Moreover, using $(u-u_0) \in C(\overline{\Omega} \times [0,T])$, we have
\begin{align*}
\|(g_{1-\al} \ast (u-u_0))(t)\|_{H^{-1}(\Omega)}
&= \dis \sup_{\|v\|_{H_0^1(\Omega)} = 1} \int_{\Omega} \left( \int_0^t g_{1-\al}(t-s) (u(x,s) - u_0(x)) \,ds \right) v(x) \,dx \\
&\leq  \|u-u_0\|_{L^{\infty}(\Omega \times (0,T))}g_{2-\al}(t)
\rightarrow 0
\end{align*}
as $t \rightarrow 0$. 
Thus, we obtain $g_{1-\al} \ast (u-u_0) \in {}_{0}H^1(0,T;H^{-1}(\Omega))$.
Therefore, $u$ is a distributional solution of (\ref{RLP}).

Conversely, assume that $u$ is the distributional solution to \eqref{RLP}.
Let  $\tilde{u} \in C(\overline{\Omega} \times [0,T])$ be the unique viscosity solution to (\ref{CP}).
As proved in the above, $\tilde{u}$ is also a distributional solution of (\ref{RLP}).
The uniqueness of distributional solutions to (\ref{RLP}) implies $u = \tilde{u}$ a.e. on $\Omega \times (0,T)$, that is, $u$ admits only one continuous representative $\tilde{u}$ on $\overline{\Omega} \times [0,T]$.
This completes the proof.
\end{proof}

%

\end{document}